\documentclass{amsart}
\usepackage{amsmath, amsthm, amssymb}
\usepackage{mathrsfs}
\usepackage{algorithm}
\usepackage{algorithmic}
\usepackage[all]{xy}

\floatname{algorithm}{Procedure}
\newcommand{\algref}[1]{Procedure~\ref{#1}}

\theoremstyle{plain}
\newtheorem{thm}{Theorem}[section]
\newtheorem{prop}[thm]{Proposition}
\newtheorem{lemma}[thm]{Lemma}

\theoremstyle{definition}

\newtheorem{ex}[thm]{Example}
\theoremstyle{remark}
\newtheorem{rem}[thm]{Remark}

\newcommand{\propref}[1]{Proposition~\ref{#1}}
\newcommand{\thmref}[1]{Theorem~\ref{#1}}
\newcommand{\lemmaref}[1]{Lemma~\ref{#1}}

\newcommand{\secref}[1]{Section~\ref{#1}}

\newcommand{\remref}[1]{Remark~\ref{#1}}
\newcommand{\tabref}[1]{Table~\ref{#1}}

\newcommand{\ZZ}{\mathbb{Z}}
\newcommand{\RR}{\mathbb{R}}
\newcommand{\CC}{\mathbb{C}}
\newcommand{\PP}{\mathbb{P}}
\newcommand{\im}[1]{\operatorname{im} (#1)}

\newcommand{\GS}[2]{\operatorname{\Gamma} (#1,#2)}
\renewcommand{\dim}[1]{\operatorname{dim} (#1)}
\renewcommand{\deg}[1]{\operatorname{deg} (#1)}

\begin{document}
\title[A method to compute Segre classes]{A method to compute Segre classes of subschemes of projective space}
\author[D. Eklund]{David Eklund}
    \address{Institut Mittag-Leffler, Aurav\"agen 17, 
SE-182 60 Djursholm Stockholm, Sweden}
    \email{daek@math.kth.se}
     \urladdr{http://www.math.kth.se/$\sim$daek}
\author[C. Jost]{Christine Jost}
    \address{Department of
             Mathematics, Stockholm University, SE-106 91 Stockholm, Sweden}
    \email{jost@math.su.se}
    \urladdr{http://www.math.su.se/$\sim$jost/}
\author[C. Peterson]{Chris Peterson}
    \address{Department of
             Mathematics, Colorado State University, Fort Collins, CO 80523}
    \email{peterson@math.colostate.edu}
    \urladdr{http://www.math.colostate.edu/$\sim$peterson}
\thanks{}
\keywords{Segre classes, computational algebraic geometry, numerical homotopy methods}
\subjclass[2000]{13Pxx, 14Qxx, 14C17, 65H10, 65E05}

\begin{abstract}
We present a method to compute the degrees of the Segre classes of a
subscheme of complex projective space. The method is based on generic
residuation and intersection theory. We provide a symbolic implementation using the software system \emph{Macaulay2} and a numerical implementation using the software package \emph{Bertini}.
\end{abstract}

\maketitle

\section{Introduction}

Segre classes are generalizations of characteristic classes of vector
bundles and they occur frequently in intersection theory. Many
problems in enumerative geometry may be solved by computing the Segre
classes of an algebraic scheme. Given an $n$-dimensional subscheme $Z$
of complex projective space there are $n+1$ Segre classes of $Z$. The
$i^{th}$ Segre class is a rational equivalence class of codimension
$i$ cycles on $Z$. Thus a Segre class may be represented as a weighted
sum of irreducible subvarieties $V_1,\dots,V_m$ of $Z$. The degree of
a Segre class is the corresponding weighted sum of the degrees of the
projective varieties $V_1,\dots,V_m$. 

In this paper we present a method to compute the degrees of the Segre
classes of $Z$, given an ideal defining $Z$. The procedure is based on
the intersection theory of Fulton and MacPherson. More specifically,
we prove a B\'ezout like theorem that involves the Segre classes of
$Z$ and a residual scheme to $Z$. The degree of the residual may be
computed providing intersection-theoretic information on the Segre
classes. This enables us to compute the degrees of these classes.

If $Z$ is smooth, the degrees of the Segre classes of $Z$ carry the
same information as the degrees of the Chern classes of the tangent
bundle of $Z$. For instance, when $Z$ is smooth, the degree of the top
Chern class, which is equal to the topological Euler characteristic of
$Z$, can be computed from the degrees of the Segre classes of
$Z$. Hence one can compute the topological Euler characteristic using
our procedure (provided that the input ideal defines a smooth
scheme). The relationship between the degrees of the Segre classes of
$Z$ and the Chern classes of the tangent bundle of $Z$ may be extended
to the non-smooth case via the so-called Chern-Fulton classes of
$Z$. In case $Z$ is smooth, the Chern-Fulton classes coincide with the
Chern classes of the tangent bundle.

We would like to mention two interesting features of our method. One
feature is that the algorithm is completely elementary in that it
requires no knowledge of intersection theory to understand the steps
in the procedure and only very basic background in algebraic
geometry. The method therefore provides a way of understanding the
computation of Segre classes from an elementary point of view. In
addition, the algorithm is easy to implement. Another feature is the
fact that our method is implementable in a numerical setting via
numerical homotopy methods, see \cite{SW} for an overview of this
area. This allows the method to be applied in settings that can be
time consuming (and even out of reach) of current symbolic
methods. One such setting is when the generating set, for an ideal
determining $Z$, has complicated coefficients. An additional setting
where current numerical methods can sometimes obtain useful
information about the degrees of Segre classes, beyond the reach of
current symbolic methods, is when $Z$ is a reduced scheme of high
codimension. These gains come through numerical approximation and
parallelization (but at the expense of exactness).

The procedure presented in this paper has been implemented in the
symbolic setting using the software system \emph{Macaulay2} \cite{GS} and in the numerical setting using the software package \emph{Bertini} \cite{BHSW}.
Both implementations are available at http://www.math.su.se/$\sim$jost/segreimplementation.htm. Initial experiments, involving subvarieties of relatively large dimension and codimension, show a great deal of promise for the algorithm in the numerical setting.

In the paper \cite{A} Aluffi formulates an algorithm that also
computes the degrees of the Segre classes of a subscheme $Z$ of
projective space. In addition he shows how to relate the computation
of the so-called Chern-Schwartz-MacPherson classes of a subscheme of
projective space to the computation of the degrees of certain Segre
classes. In the present paper we present an alternative method to
Aluffi's. Though the two are closely related, they have a rather
different computational behavior and seem to complement each other
well (see \secref{sec:benchmarks} for some examples). Apart from the
difference in computing speed in various cases one may ask what is the
need for another method with the same output as an existing
method. One answer is that our approach is different and therefore
sheds new light on the problem of computing Segre classes. But more
importantly we would answer by repeating the two features mentioned
above, namely that our method is elementary and that it is readily
amenable to numerical computation.

The paper is organized as follows. In \secref{sec:background} we give
the basic definitions and state a theorem from intersection theory. In
\secref{sec:computing} we derive a recursive formula for Segre classes
which is the basis of our method. The procedure to compute Segre
classes is presented in \secref{sec:method}. Some examples are given
in \secref{sec:examples} and in \secref{sec:benchmarks} we give a list
of run times on examples comparing our method to other algorithms.

The results of this paper are generalizations and variants of the
results in \cite{BEP,DEPS} to the setting of subschemes of projective space.

\section*{Acknowledgments}
We would like to thank Paolo Aluffi and Sandra Di Rocco for their useful comments and
encouragement. We thank Jon Hauenstein for pointing out how we
could utilize efficient numerical methods and also for his help with
running examples in \emph{Bertini}. Finally, we thank the Institut Mittag-Leffler for
their wonderful research environment that facilitated the completion of this paper.

\section{Background in intersection theory} \label{sec:background}
We start by going through some concepts and results from intersection
theory. For this paper, the main reference on matters of intersection
theory is Fulton's book \cite{F}. 

\subsection{Notation}
Let $Y$ be an algebraic scheme over $\CC$ of dimension $n$. By a
subscheme of $Y$ we will mean a closed subscheme. We will denote by
$C_p(Y)$ the free Abelian group on irreducible $p$-dimensional
subvarieties of $Y$. The $p^{th}$ Chow group of $Y$ is the quotient of
$C_p(Y)$ by the cycles rationally equivalent to 0, and it is denoted
$A_p(Y)$. The Chow group of $Y$ is the group $A_*(Y)=\bigoplus_{p=0}^n
A_p(Y)$. Given an element $\alpha \in A_*(Y)$, $\{\alpha\}_p \in
A_p(Y)$ will denote the $p^{th}$ homogeneous component of $\alpha$ (if
$n<p$, then $\{\alpha\}_p=0$). A subscheme $X \subseteq Y$ induces a
cycle class $[X] \in A_*(Y)$ represented by $\sum_{i=1}^t m_i X_i$,
where $X_1,\dots,X_t$ are the irreducible components of $X$ and
$m_1,\dots,m_t$ their geometric multiplicities in $X$. In particular,
$[\emptyset]=0$. If $\alpha \in A_*(X)$, we will at times consider
$\alpha$ to be an element of $A_*(Y)$, omitting in the notation
the push-forward under the inclusion map.

For a rank $\rho$ vector bundle $E$ on $Y$ we have the Chern class
operations $c_i(E):A_p(Y) \rightarrow A_{p-i}(Y)$ for $i \leq p$, $0
\leq i \leq \rho$ and $0 \leq p \leq n$, see \cite{F} Chapter~3. The
value of $c_i(E)$ on $\alpha \in A_p(Y)$ is denoted $c_i(E) \cap
\alpha$. The corresponding map $A_*(Y) \rightarrow A_*(Y)$ is also
denoted $c_i(E)$, where $c_i(E)\cap \alpha=0$ if $\alpha \in A_p(Y)$
and $p < i$. The total Chern class operation $c(E):A_*(Y) \rightarrow
A_*(Y):\alpha \mapsto c(E)\cap \alpha$ is defined by $c(E)=\sum_i
c_i(E)$. The operation of a product of Chern classes on the Chow group is defined as the composition of
the individual Chern class operations. The map $c_0(E)$ is the identity
homomorphism. If $Y$ is smooth, the Chern classes are well defined
elements of $A_*(Y)$ and the operations $c_i(E) \cap \alpha$
correspond to the intersection product. Let $X \subseteq Y$ be a
closed subscheme with inclusion $i:X \rightarrow Y$ and let $\alpha
\in A_*(X)$. By the notational convention mentioned above we will
sometimes write $c(E)\cap \alpha$ to mean $c(E) \cap i_*(\alpha) \in
A_*(Y)$. By the projection formula, see \cite{F} Theorem~3.2~(c),
$c(E) \cap i_*(\alpha)=i_*(c(i^*E)\cap \alpha)$.

For a Cartier divisor $D$ on $Y$, the corresponding line bundle on $Y$
is denoted by $\mathcal{O}_Y(D)$.

In this paper, varieties are by definition irreducible and
reduced. Finally, we use the convention $\dim{\emptyset}=-1$.

\subsection{Regular embeddings}
Let $Y$ be an algebraic scheme over $\CC$. A closed embedding $X
\rightarrow Y$ of a subscheme $X$ of $Y$ is called a regular embedding
of codimension $d$ if the following holds. Every point of $X$ has an
affine open neighborhood $U$ in $Y$ such that the ideal defining $X
\cap U$ is generated by a regular sequence of length $d$ in the
coordinate ring of $U$. If, for some $d$, $X$ is a regular embedding in $Y$ of codimension $d$, then we will simply
say that the embedding is regular.

\begin{lemma} \label{lemma:regular}
Let $Y$ be a complex variety and let $\mathcal{L}$ be a line bundle on
$Y$ such that the corresponding complete linear system is base point
free. For global sections $\sigma_1,\dots,\sigma_{\mu} \in \GS{Y}{\mathcal{L}}$
of $\mathcal{L}$, let $X=X(\sigma_1,\dots,\sigma_{\mu})$ denote the scheme of
common zeros of $\sigma_1,\dots,\sigma_{\mu}$. Then, for general
$\sigma_1,\dots,\sigma_{\mu} \in \GS{Y}{\mathcal{L}}$, the natural embedding $X
\rightarrow Y$ is regular.
\end{lemma}
\begin{proof}
Let $\phi:Y \rightarrow \PP^r$ be the map given by the complete linear
system corresponding to $\mathcal{L}$ and put $e=\dim{\im{\phi}}$. If
$\mu > e$ then $X(\sigma_1,\dots,\sigma_{\mu})=\emptyset$ for general sections
$\sigma_1,\dots,\sigma_{{\mu}}$ and hence we may assume that $\mu \leq e$. That
$X(\sigma_1) \rightarrow Y$ is a regular embedding for a general section
$\sigma_1$ is clear since $\mathcal{O}_Y(U)$ is an integral domain for any
open set $U \subseteq Y$. Hence, for any affine open $U \subseteq Y$,
$\mathcal{O}_Y(U)$ has no zero divisors and any $\sigma_1 \neq 0$ will give
a regular sequence in $\mathcal{O}_Y(U)$. If $\mu=1$ we are done. If
$\dim{Y} \leq 1$ then $e \leq 1$ and therefore $\mu=1$ in this
case. Assume that $1<\mu$, in particular $1<\dim{Y}$. Then $e \geq 2$
and it follows from Bertini type theorems that $X(\sigma_1)$ is a variety
for a general section $\sigma_1$. In fact, $X(\sigma_1)$ is reduced by
\cite{FOV} Corollary~3.4.9 and $X(\sigma_1)$ is irreducible by \cite{FOV}
3.4.10. Replacing $Y$ by $X(\sigma_1)$ and restricting $\mathcal{L}$ to
$X(\sigma_1)$ we have reduced to the case of a lower dimensional ambient
variety since the composition of two regular embeddings is regular
(see \cite{F} Appendix~B.7.4).
\end{proof}

\subsection{Segre classes and intersection products}
Let $Y$ be a complex variety and let $X$ be an $n$-dimensional
subscheme of $Y$. Suppose that $X \neq Y$ and let $\widetilde{Y}$ be
the blow-up of $Y$ along $X$. Let $\pi: \widetilde{Y} \rightarrow Y$
be the projection, let $\widetilde{X}=\pi^{-1}(X)$ be the exceptional
divisor, and let $\eta=\pi |_{\widetilde{X}}$. The total Segre class
$s(X,Y)$ of $X$ in $Y$ is an element of $A_*(X)$ which may be
characterized as follows (see \cite{F}
Corollary~4.2.2): $$s(X,Y)=\sum_{p\geq
  1}(-1)^{p-1}\eta_*(\widetilde{X}^p).$$ Here $\widetilde{X}^p$ is the
self intersection of Cartier divisors defined in \cite{F}
Definition~2.4.2.

\begin{rem} The Cartier divisor $\widetilde{X}$ on
$\widetilde{Y}$ corresponds to a line bundle on $\widetilde{Y}$ whose
  restriction to $\widetilde{X}$ is the normal bundle of
  $\widetilde{X}$ in $\widetilde{Y}$. The dual bundle to the normal
  bundle is denoted $\mathcal{O}(1)$. The exceptional divisor
  $\widetilde{X}$ is naturally identified with the so-called
  projective normal cone $P(C_{X}Y)$ of $X$ in $Y$. The total Segre
  class is given by $$s(X,Y)=\sum_{i\geq 0}
  \eta_*(c_1(\mathcal{O}(1))^i \cap [P(C_{X}Y)]).$$ This definition of
  Segre classes generalizes beyond normal cones of subschemes to
  arbitrary cones, see \cite{F} Chapter~4.
\end{rem}

We will now recall the Fulton-MacPherson approach to intersection
products, see \cite{F} Chapter~6. Let $Y$ and $V$ be complex varieties
and put $k=\dim{V}$. Let $X$ be a closed subscheme of $Y$ such that
there is a regular embedding $i:X \rightarrow Y$ of codimension $d
\leq k$. Let $f:V \rightarrow Y$ be a morphism and put
$W=f^{-1}(X)$. Then we get the following fibre product diagram
\begin{center}
$\xymatrix{W \ar[r]^j \ar[d]_g & V \ar[d]^{f} \\ X \ar[r]_i & Y}$
\end{center}
where $j:W \rightarrow V$ is the inclusion and $g:W \rightarrow X$ the
restriction of $f$ to $W$. The normal cone $C_WV$ of $W$ in $V$ can be
constructed as follows. Suppose first that $V$ is affine with
coordinate ring $A$ and that $W$ is defined by an ideal $J \subseteq
A$ generated by $f_1,\dots,f_d \in A$. Let $B=A/J$. Then $C_W V$ is
the spectrum of the $B$-algebra $\bigoplus_{s \geq 0}
J^s/J^{s+1}$. Thus $C_WV$ may be embedded as a closed subscheme of $W
\times \CC^d$ defined by the kernel of the surjective
homomorphism $$B[x_1,\dots,x_d] \rightarrow \bigoplus_{s \geq 0}
J^s/J^{s+1}$$ which maps $x_i$ to the image of $f_i$ in $J/J^2$. In
the general case, $C_W V$ may be constructed by covering $V$ with open
affine subsets and gluing the normal cones of the affine patches
together. In case the embedding $W \rightarrow V$ is regular, $C_WV$
is a vector bundle, namely the normal bundle $N_WV$. In this case the
total Segre class of $W$ in $V$ is the inverse of the total Chern
class of $N_WV$ in the sense that $s(W,V)=c(N_WV)^{-1}\cap [W]$, where
$c(N_WV)^{-1}$ is a formal inverse of $c(N_WV)$. The normal cone
$C_WV$ has pure dimension $k$, see \cite{F} Appendix~B.6.6. Let
$C=C_WV$, let $N=g^*(N_XY)$ (where $N_XY$ is the normal bundle of $X$
in $Y$), and let $p:N \rightarrow W$ be the projection. It is shown
in \cite{F} Chapter~6, that $C$ embeds in $N$ and therefore it
determines a class $[C] \in A_k(N)$. Now, by \cite{F} Theorem 3.3 (a),
the flat-pullback $p^*: A_{k-d}(W) \rightarrow A_k(N)$ is an
isomorphism. This map is given by $p^*(Z)=[p^{-1}(Z)]$ for an
irreducible subvariety $Z \subseteq W$. The intersection product of
$V$ by $X$ on $Y$ is a class in $A_{k-d}(W)$ denoted $X\cdot V$ and
defined by $$X\cdot V=(p^*)^{-1}([C]).$$ An important connection to
Segre classes is given by \cite{F} Proposition~6.1~(a): $$X\cdot
V=\{c(N)\cap s(W,V)\}_{k-d}.$$

We will now state the Residual Intersection Formula from \cite{F}
which is the main result underpinning our method to compute Segre
classes. Let $Y$, $X$, $W$, $V$, $N$, $k$ and $d$ be as in the above
definition of the intersection product $X \cdot V$. Let $Z \subseteq
W$ be a closed subscheme and suppose that $Z \neq V$. Let
$\pi:\widetilde{V} \rightarrow V$ be the blow-up of $V$ along $Z$ and
put $\widetilde{W}=\pi^{-1}(W)$ and $\widetilde{Z}=\pi^{-1}(Z)$. Let
$\widetilde{R}$ be the residual scheme to $\widetilde{Z}$ in
$\widetilde{W}$ with respect to $\widetilde{V}$, see
\cite{F}~Definition 9.2.1. This is a scheme such that, if
$\mathscr{I}(\widetilde{Z})$, $\mathscr{I}(\widetilde{W})$ and
$\mathscr{I}(\widetilde{R})$ are the ideal sheaves in
$\mathcal{O}_{\widetilde{V}}$ defining the respective schemes,
then $$\mathscr{I}(\widetilde{W})=\mathscr{I}(\widetilde{Z}) \cdot
\mathscr{I}(\widetilde{R}).$$ Let $\eta: \widetilde{W} \rightarrow W$
be the restriction of $\pi$ to $\widetilde{W}$ and let
$\mathcal{O}(-\widetilde{Z})$ denote the pullback of
$\mathcal{O}_{\widetilde{V}}(-\widetilde{Z})$ under the inclusion
$\widetilde{W} \rightarrow \widetilde{V}$.

The following proposition is Corollary~9.2.3 of \cite{F}.
\begin{prop} \label{prop:Fulton}
With notation as above, $$X \cdot V =\{c(N)\cap s(Z,V)\}_{k-d}+\RR,$$
where $\RR=\eta_*(\{c(\eta^*N\otimes
\mathcal{O}(-\widetilde{Z}))\cap
s(\widetilde{R},\widetilde{V})\}_{k-d})$.
\end{prop}

\section{Computing Segre classes of projective schemes} \label{sec:computing}
Let $Z$ be a proper $n$-dimensional subscheme of complex projective
space $\PP^k$. This paper is about a method for computing the
push-forward of $s(Z,\PP^k)$ to $\PP^k$ given an ideal defining
$Z$. In this section we explain how to derive information about the
push-forward given sufficiently general elements from the ideal.

Let $s_0,\dots,s_n$ be the homogeneous components of $s(Z,\PP^k)$ with
$s_i$ of codimension $i$, that is $s(Z,\PP^k)=\sum_{i=0}^n s_i$ with
$s_i \in A_{n-i}(Z)$. The degree of a 0-cycle $\alpha=\sum_i m_i p_i$
on $\PP^k$, $m_i \in \ZZ$ and $p_i \in \PP^k$, is simply
$\deg{\alpha}=\sum_i m_i$. If $\gamma:Z \rightarrow \PP^k$ is the
inclusion map, we define the degree of $s_i$
by $$\deg{s_i}=\deg{\gamma_*(s_i)\cdot H^{n-i}},$$ where $H \in
A_{k-1}(\PP^k)$ is the hyperplane class and the product is the
intersection product on $\PP^k$. The numbers $\{\deg{s_i}\}_i$ is the output
of our procedure, they carry the same information as the push-forward
$\gamma_*(s(Z,\PP^k))$. The degree of any $\alpha \in A_p(\PP^k)$ is
defined similarly by $\deg{\alpha \cdot H^p}$.

Let $I \subseteq \CC[x_0,\dots,x_k]$ be a homogeneous ideal. For a
positive integer $m$, we use $I(m)$ to denote the $m^{th}$ graded
piece of $I$. Given a homogeneous ideal $J$, the ideal quotient $J:I$
is given by $$J:I=\{f\in \CC[x_0,\dots,x_k]:fI\subseteq J\},$$ and the
saturation of $J$ with respect to $I$ is $$J:I^{\infty}=\bigcup_{p\geq
  1} J:I^p.$$ Note that $(J:I^p):I=J:I^{p+1}$ for $p \geq 1$, that
the ascending sequence of ideal quotients $J:I \subseteq J:I^2
\subseteq J:I^3 \subseteq \dots$ stabilizes and that $J:I^p=J:I^{\infty}$
for large enough $p$.

\begin{rem}
Let $I$ and $J$ be homogeneous ideals of $\CC[x_0,\dots,x_k]$ and let
$V(I)$ and $V(J)$ denote the corresponding zero-loci in $\PP^k$. If
$I=\CC[x_0,\dots,x_k]$, then $J:I^{\infty}=J$. Suppose $I \neq
\CC[x_0,\dots,x_k]$. The ideal $J:I^{\infty}$ is homogeneous and the
scheme $R$ defined by $J:I^{\infty}$ is supported on the
Zariski-closure of $V(J) \setminus V(I)$. In fact, if $J=\bigcap_i
Q_i$ is a primary decomposition (so each $Q_i$ is homogeneous and primary)
then 
\begin{equation} \label{eq:primary}
J:I^{\infty}=\bigcap_{\{i\hskip 1pt : \hskip 1pt V(Q_i)
  \nsubseteq V(I)\}} Q_i.
\end{equation}
To see this, note that $J:I^{\infty}=\bigcap_i (Q_i:I^{\infty})$ and
that $V(Q_i) \subseteq V(I)$ precisely when $\sqrt{Q_i} \supseteq
I$. If $\sqrt{Q_i} \supseteq I$, then $Q_i \supseteq I^p$ for some $p$
and $Q_i:I^{\infty}= \CC[x_0,\dots,x_k]$. On the other hand, if
$\sqrt{Q_i} \nsupseteq I$, then $Q_i:I=Q_i$ since $Q_i$ is primary. It
follows that $Q_i:I^{\infty}=Q_i$ in this case.
\end{rem}

The following theorem is a B\'ezout like equality which gives rise
to a recursive formula for the degrees of the Segre classes of $Z$ in
$\PP^k$. Using the statement of the theorem, we may express the degree
of a Segre class $s_p$ in terms of $\deg{s_i}$ for $i < p$ and the
degree of a certain residual scheme $R$. Computing the degree of the
residual $R$ is the main computational step in the method.

\begin{thm} \label{thm:main}
Let $Z \subset \PP^k$ be a subscheme of dimension $n$ defined by a
non-zero homogeneous ideal $I \subseteq \CC[x_0,\dots,x_k]$. Let
$s_0,\dots,s_n$ denote the Segre classes of $Z$ in $\PP^k$. Let
$g_0,\dots,g_r$ be a set of non-zero homogeneous generators of $I$ and
put $m=\max_i \{\deg{g_i}\}$. For $k-n \leq d \leq k$ and general
elements $f_1,\dots,f_d \in I(m)$, the following holds. If $J$ is the
ideal generated by $\{f_1,\dots,f_d\}$ and $R \subseteq \PP^k$ is the
subscheme defined by $J:I^{\infty}$, then
$$m^d=\deg{R}+\sum_{i=0}^p\binom{d}{p-i}m^{p-i}\deg{s_i},$$
where $p=d-(k-n)$.
\end{thm}
\begin{proof}
The proof is divided in steps 0 through 4.

Step 0: setup. Let $I'$ be the ideal generated by $I(m)$ and let
$\mathfrak{m}=(x_0,\dots,x_k)$. Then
$I:\mathfrak{m}^{\infty}=I':\mathfrak{m}^{\infty}$, and therefore $I$
and $I'$ define the same scheme $Z \subseteq \PP^k$. We may thus
assume that $g_0,\dots,g_r$ all have degree
$m$. Let $$\pi:\widetilde{\PP}^k \rightarrow \PP^k$$ be the blow-up of
$\PP^k$ along $Z$ and put $\widetilde{Z}=\pi^{-1}(Z)$. The map $\PP^k
\setminus Z \rightarrow \PP^r$ defined by $g_0,\dots,g_r$ extends to a
map $$\phi: \widetilde{\PP}^k \rightarrow \PP^r,$$ see
\cite{F}~4.4. In fact, $\widetilde{\PP}^k$ embeds in $\PP^k \times
\PP^r$ in such a way that $(\widetilde{\PP}^k \setminus
\widetilde{Z})$ is the graph of the map $\PP^k \setminus Z \rightarrow
\PP^r$ and $\phi$ is the projection. Let $W \subseteq \PP^k$ be the
scheme defined by $f_1,\dots,f_d$ and put
$\widetilde{W}=\pi^{-1}(W)$. Let $\widetilde{R}$ be the residual to
$\widetilde{Z}$ in $\widetilde{W}$ with respect to
$\widetilde{\PP}^k$.

Step 1: we will show that $\widetilde{R} \rightarrow
\widetilde{\PP}^k$ is a regular embedding and that $\widetilde{R}$ is
either empty or of pure dimension $k-d$ and that no irreducible
component of $\widetilde{R}$ is contained in $\widetilde{Z}$. By
\cite{F}~4.4., $$\phi^*(\mathcal{O}_{\PP^r}(1))=\pi^*(\mathcal{O}_{\PP^k}(m))
\otimes \mathcal{O}_{\widetilde{\PP}^k}(-\widetilde{Z}).$$ In concrete
terms, $\widetilde{\PP}^k$ is defined by a bi-homogeneous ideal $$K
\subseteq \mathbb{C}[x_0,\ldots,x_k][y_0,\ldots, y_r]$$ such that $K$
contains the elements $g_iy_j-g_jy_i$ for $0\leq i< j\leq r$. Observe
that $\widetilde{Z}$ is given by the vanishing of $g_0,\ldots,
g_r$. Consider the affine open set $$U=U_{\alpha\beta} =
\{(x_0,\dots,x_k,y_0,\dots,y_r) \in \widetilde{\PP}^k:x_\alpha \neq 0,
y_\beta \neq 0\}$$ and let $w_0 = \frac{y_0}{y_\beta},\dots,w_r =
\frac{y_r}{y_\beta}$ with $w_\beta = 1$. Then
$w_0,\dots,\widehat{w_{\beta}},\dots,w_r$ are coordinates on $\CC^r =
\{y_\beta \neq 0\} \subset \PP^r$. Note that for all $i$,
$f_i=\sum_{j=0}^r \lambda_i^j g_j$, for a general vector
$(\lambda_i^0,\dots,\lambda_i^r) \in \CC^{r+1}$. With an abuse of
notation, we use $g_i$, $f_i$ and $w_i$ to denote the corresponding
elements of the coordinate ring of $U$. Then, $g_j = w_j g_\beta$ for
all $j$. Hence $\widetilde{Z} \cap U$ is defined by $g_\beta$. Also, $
f_i=(\sum_{j=0}^r \lambda_i^j w_j) g_\beta$. It follows that
$\widetilde{R} \cap U$ is defined by the ideal $(\sum_{j=0}^r
\lambda_1^j w_j, \ldots, \sum_{j=0}^r \lambda_d^j w_j)$. We conclude
that $\widetilde{R}=\phi^{-1}(L)$ for a general linear subspace $L
\subseteq \PP^r$ of codimension $d$ (if $r < d$, then
$\widetilde{R}=\emptyset$). Hence $\widetilde{R}$ is either empty or
of pure dimension $k-d$ and $\widetilde{R} \cap \widetilde{Z}$ is
either empty or of pure dimension $k-d-1$. It follows that no
irreducible component of $\widetilde{R}$ is contained in
$\widetilde{Z}$. Since $\widetilde{\PP}^k$ is a variety (see \cite{H}
Proposition~II.7.16 or \cite{F} Appendix~B.6.4), it follows by
\lemmaref{lemma:regular} that the embedding of $\widetilde{R}$ in
$\widetilde{\PP}^k$ is regular.

Step 2: applying \propref{prop:Fulton}. Let $X_1,\dots,X_d$ be defined
by $X_{\nu}=\{f_{\nu}=0\}$. Then $W=\bigcap_{\nu=1}^d X_{\nu}$. Let
$X=X_1 \times \dots \times X_d$ and $Y=\PP^k \times \dots \times
\PP^k$ ($d$ factors). Let $j:W \rightarrow \PP^k$ be the inclusion and
let $f:\PP^k \rightarrow Y$ and $g:W \rightarrow X$ be the diagonal
morphisms. The morphism $i:X \rightarrow Y$ induced by the inclusions
$X_1,\dots,X_d \subset \PP^k$ is a regular embedding of codimension
$d$. Put $N=g^*(N_XY)$. Letting $V=\PP^k$, we apply
\propref{prop:Fulton} to the diagram
$$\xymatrix{\bigcap_{\nu=1}^d X_{\nu} \ar[r]^j \ar[d]_g & \PP^k \ar[d]^{f} \\ 
X_1 \times \dots \times X_d \ar[r]_i & \PP^k \times \dots \times
  \PP^k}$$ and conclude that 
\begin{equation} \label{eq:Fulton}
X \cdot \PP^k=\{c(N)\cap s(Z,\PP^k)\}_{k-d}+\RR.
\end{equation}
Here $s(Z,\PP^k)$ is regarded as a rational equivalence class on $W$.

Step 3: we will show
that $$m^d=\sum_{i=0}^p\binom{d}{p-i}m^{p-i}\deg{s_i}+\deg{j_*(\RR)},$$
where $p=d-(k-n)$. We shall first see that $N=j^*(E)$ where
$E=\bigoplus_{\nu=1}^d \mathcal{O}_{\PP^k}(m)$.  Let $p_{\nu}:Y
\rightarrow \PP^k$ be the $\nu^{th}$ projection and let
$\mathbb{X}_\nu$ be the divisor $p_{\nu}^{-1}(X_{\nu})$ on $Y$. Then,
by \cite{F} B.7.4,
$$N_XY=\bigoplus_{\nu=1}^d\mathcal{O}_Y({\mathbb{X}_{\nu}})|_X.$$ Since
$f^*(\mathcal{O}_Y(\mathbb{X}_{\nu}))=\mathcal{O}_{\PP^k}(m)$ for all
$\nu$, we have that $$j^*(\mathcal{O}_{\PP^k}(m))=(f\circ
j)^*(\mathcal{O}_Y({\mathbb{X}_{\nu}}))=(i\circ
g)^*(\mathcal{O}_Y({\mathbb{X}_{\nu}}))=g^*(\mathcal{O}_Y({\mathbb{X}_{\nu}})|_X)$$
for all $\nu$. Hence
$$N=g^*(N_XY)=\bigoplus_{\nu=1}^dg^*(\mathcal{O}_Y({\mathbb{X}_{\nu}})|_X)=j^*E.$$
Note that $c(E)=(1+mH)^d \in A_*(\PP^k)$, where $H \in A_{k-1}(\PP^k)$
is the hyperplane class.

We now push both sides of (\ref{eq:Fulton}) forward to $\PP^k$ by $j$
and then take degrees. By B\'ezout's theorem, $\deg{j_*(X \cdot
  \PP^k)}=m^d$ (see \cite{F} Example~6.2.6). By the projection
formula, $j_*(c(N)\cap s(Z,\PP^k))=c(E) \cap j_*(s(Z,\PP^k))$. We get
that $$j_*(\{c(N)\cap s(Z,\PP^k)\}_{k-d})=\{j_*(c(N)\cap
s(Z,\PP^k))\}_{k-d}=$$ $$\{(1+mH)^d \cdot
j_*(s(Z,\PP^k))\}_{k-d}=\{(1+mH)^d\cdot \sum_{i=0}^n
j_*(s_i)\}_{k-d}.$$ The degree of the latter expression
is $$\sum_{i=0}^p\binom{d}{p-i}m^{p-i}\deg{s_i},$$ where $p=d-(k-n)$.

Step 4: it remains to see that $\deg{j_*(\RR)}=\deg{R}$ where $R
\subseteq \PP^k$ is the scheme defined by $J:I^{\infty}$. In fact, we
shall see that $j_*(\RR)=[R]$ in $A_*(\PP^k)$. Let $\eta:\widetilde{W}
\rightarrow W$ be the restriction of $\pi$ to $\widetilde{W}$. Since
$\widetilde{R} \rightarrow \widetilde{\PP}^k$ is a regular embedding we
have that
$s(\widetilde{R},\widetilde{\PP}^k)=c(N_{\widetilde{R}}\widetilde{\PP}^k)^{-1}
\cap [\widetilde{R}]$. Since $\widetilde{R}$ is either empty or has
pure dimension $k-d$, $$\{c(\eta^*N\otimes
\mathcal{O}(-\widetilde{Z}))\cap
s(\widetilde{R},\widetilde{\PP}^k)\}_{k-d}=$$ $$\{c(\eta^*N\otimes
\mathcal{O}(-\widetilde{Z}))\cap
(c(N_{\widetilde{R}}\widetilde{\PP}^k)^{-1} \cap
[\widetilde{R}])\}_{k-d}=[\widetilde{R}].$$ Hence
$\RR=\eta_*([\widetilde{R}])$. Let
$\widetilde{R}_1,\dots,\widetilde{R}_t$ be the irreducible components
of $\widetilde{R}$ and let $m_1,\dots,m_t$ denote their geometric
multiplicities. Since none of the components $\widetilde{R}_1,\dots,\widetilde{R}_t$ is contained in
$\widetilde{Z}$ and $\pi:\widetilde{\PP}^k \rightarrow \PP^k$ is an
isomorphism outside $\widetilde{Z}$, $j_*(\RR)=\sum_{i=1}^t m_i
[\pi(\widetilde{R}_i)]$. Observe that $\pi$ induces an isomorphism 
$(\widetilde{R} \setminus \widetilde{Z}) \cong (W \setminus Z)$. It follows from (\ref{eq:primary}) that $\sum_{i=1}^t m_i [\pi(\widetilde{R}_i)]=[R]$.
\end{proof}

\begin{rem}
With notation as in the proof of \thmref{thm:main}, note that the
group $\textrm{GL}(\CC^{k+1}) \times \textrm{GL}(\CC^{r+1})$ acts
transitively on $\PP^k \times \PP^r$ and that $\widetilde{\PP}^k$ is
regular outside $\widetilde{Z}$. It follows by Kleiman's
transversality theorem \cite{K} that for a general linear subspace $L
\subseteq \PP^r$ of codimension $d$, $\widetilde{R}=\widetilde{\PP}^k
\cap (\PP^{k} \times L)$ is regular outside $\widetilde{Z}$ and the
multiplicities $m_1,\dots,m_t$ of the components of $\widetilde{R}$
are all equal to 1. Moreover, it follows that the scheme $R$ defined
by $J:I^{\infty}$ is regular outside $Z$. This could be of interest in
connection with the computation of the degree of $R$, which is the
main computational ingredient in our method to compute Segre classes.
\end{rem}

\section{The method} \label{sec:method}
\thmref{thm:main} states that certain conditions hold for a
\emph{general} choice of elements of a given ideal. By choosing these
elements randomly we turn this into a probabilistic
algorithm. Applying \thmref{thm:main} to solve for the Segre classes
recursively, we obtain the following procedure to compute the degrees
of the Segre classes of a subscheme of projective space. The input is
an ideal defining the subscheme.

\begin{algorithm}[h] 
\caption{A method to compute the degrees of Segre classes}
\label{alg:segre}
\begin{algorithmic}[1]
\REQUIRE Non-zero homogeneous generators $g_0,\dots,g_r$ of an ideal $I \subseteq \CC[x_0,\dots,x_k]$.  
\ENSURE The degrees of the Segre classes of the subscheme of $\PP^k$ defined by $I$.
\medskip
\STATE Let $m=\max_i \{\deg{g_i}\}$.
\STATE Let $Z \subset \PP^k$ be the scheme defined by $I$ and compute $n=\dim{Z}$.
\STATE Pick random elements $f_1,\dots,f_k \in I(m)$.
\FOR {$d = k-n$ to $k$}
\STATE Let $J=(f_1,\dots,f_d)$.
\STATE Compute $\deg{R}$, where $R \subseteq \PP^k$ is the scheme defined by $J:I^{\infty}$.
\STATE Let $p=d-(k-n)$ and compute $$\deg{s_p}=m^d-\deg{R}-\sum_{i=0}^{p-1}\binom{d}{p-i}m^{p-i}\deg{s_i}.$$
\ENDFOR
\RETURN $\deg{s_0},\dots,\deg{s_n}$
\end{algorithmic}
\end{algorithm}

\begin{rem} \label{rem:improvement}
We will use the notation of \algref{alg:segre} and the proof of
\thmref{thm:main}. In particular $\pi: \widetilde{\PP}^k \rightarrow
\PP^k$ denotes the blow-up of $\PP^k$ along $Z$. Instead of saturating
with respect to the whole ideal $I$, as is done in \algref{alg:segre},
one could saturate with respect to one element of $I$. Let $h \in I$,
$h\neq 0$, let $R'$ be the scheme defined by $J:(h)^{\infty}$ and let
$H \subseteq \PP^k$ be the hypersurface defined by $h$. The claim is
that we could replace $R$ by $R'$ in \algref{alg:segre}. Tracing back
the conditions on $R$ used in the proof of \thmref{thm:main} we see
that we only need to show that no irreducible component of the
residual $\widetilde{R} \subseteq \widetilde{\PP}^k$ is contained in
$\pi^{-1}(H)$. This follows exactly as in step~1 of the proof of
\thmref{thm:main} where it is shown that $\widetilde{R}$ has no
irreducible component inside the exceptional divisor $\widetilde{Z}$.
\end{rem}

\begin{rem} \label{rem:chern-fulton}
Let $Z \subseteq \PP^k$ be a subscheme of dimension $n$ and let $s_i
\in A_{n-i}(Z)$ be the Segre classes of $Z$, that is
$s(Z,\PP^k)=\sum_{i=0}^n s_i$. Define the total Chern-Fulton class of
$Z$ by $$c'(Z)=c(T_{\PP^k}|_Z) \cap s(Z,\PP^k),$$ where $T_{\PP^k}$ is
the tangent bundle of $\PP^k$. This definition is independent of the
embedding of $Z$ in $\PP^k$ in the sense that if $Z$ admits two
embeddings into smooth varieties $M$ and $P$, then $c(T_{M}|_Z) \cap
s(Z,M)=c(T_{P}|_Z) \cap s(Z,P)$, see \cite{F} Example~4.2.6. Let
$c'(Z)=\sum_{i=0}^n c'_i$, with $c'_i \in A_{n-i}(Z)$. If $Z$ is
smooth, then the Chern-Fulton classes coincide with the Chern classes
of the tangent bundle. Since $c(T_{\PP^k})=(1+H)^{k+1}$ where $H \in
A_{k-1}(\PP^k)$ is the hyperplane class, the degrees of the Segre
classes and those of the Chern-Fulton classes are related
by $$\deg{c'_i}=\sum_{p=0}^i \binom{k+1}{i-p}\deg{s_p}.$$ Recall that
in the smooth case, the degree of the top Chern class of the tangent
bundle is equal to the topological Euler characteristic. Thus, in case
$Z$ is smooth, $\deg{c'_n}$ is the topological Euler characteristic of
$Z$ and \algref{alg:segre} provides a way of computing this
topological invariant.
\end{rem}

\section{Examples} \label{sec:examples}
In this section we illustrate \algref{alg:segre} with some
examples. In these examples, for an $n$-dimensional subscheme $Z
\subseteq \PP^k$, we use the notation
$\sigma(Z)=(\deg{s_0},\dots,\deg{s_n})$, where
$s(Z,\PP^k)=\sum_{i=0}^n s_i$ and $s_i \in A_{n-i}(Z)$.

\begin{ex}
Let $I \subset \CC[x,y,z]$ be the ideal $I=(x^2,y^2,xy)$ and let $p
\in \PP^2$ be the degree three zero-scheme defined by $I$. Then $s(p,\PP^2)=s_0 \in
A_0(p)$, and $A_0(p) \cong \ZZ$ via the degree map. Now let $f_1, f_2
\in I$ be general elements of degree 2 and put $J=(f_1,f_2)$. Then
$J:I=(x,y)$ and $(J:I):I=(1)$. Hence $J:I^{\infty}=(1)$ and the
residual $R$ is empty. Therefore $$s_0=2^2 - \deg{R}=4.$$ Note that
$s_0$ is not equal to the degree of $p$ (which is equal to 3).
\end{ex}

\begin{ex}
Consider a plane curve $C \subseteq \PP^2$ defined by one element $g
\in \CC[x,y,z]$ of degree $m$. Then it is immediate from
\algref{alg:segre} that $\sigma(C)=(m,-m^2)$. In case $g=xy$, we get
$\sigma(C)=(2,-4)$.

Now consider the scheme $D \subseteq \PP^2$ defined by
$I=(x^2y,xy^2)$. The support of $D$ is the union of the lines
$L_1=\{x=0\}$ and $L_2=\{y=0\}$, but $D$ has an embedded point at
$p=L_1\cap L_2$. The normal cone $C_D \PP^2$ has three irreducible
components, all of dimension 2, and the supports of these components
are $L_1$, $L_2$ and $p$, respectively. The supports are the so-called
distinguished varieties of the intersection defined by
$\{x^2y,xy^2\}$. A general element $f_1 \in I(3)$ may be written
$f_1=xy(ax+by)$, for general $a,b \in \CC$. Then
$(f_1):I^{\infty}=(f_1):I=(ax+by)$. Hence the residual $R$ is the line
$\{ax+by=0\}$ and $$\deg{s_0}=3-\deg{R}=2.$$ For a general $f_2 \in
I(3)$ we have that $(f_1,f_2)=I$ and it follows
that $$\deg{s_1}=3^2-2\cdot3\deg{s_0}=-3.$$ In summary,
$\sigma(D)=(2,-3)$. Observe that the Segre classes detect the embedded
point $p$.
\end{ex}

\begin{ex}
To illustrate \algref{alg:segre} we show in this example how it works
on a surface $Z \subseteq \PP^k$. Let $Z$ be defined by an ideal $I$
which is generated by polynomials $g_0,\dots,g_r$ in
$\CC[x_0,\dots,x_k]$. Let $m=\max_i\{\deg{g_i}\}$ and let $f_1,\dots,f_k
\in I(m)$ be general. Let $J_2=(f_1,\dots,f_{k-2})$,
$J_1=(f_1,\dots,f_{k-1})$ and $J_0=(f_1,\dots,f_k)$. Let $W_i$ be the
scheme defined by $J_i$ and let $R_i$ be the scheme defined by
$J_i:I^{\infty}$. Then
\begin{align*}
W_2=Z\cup R_2 \quad \dim{R_2}=2 \; \textrm{(or $R_2$ is empty)},\\
W_1=Z\cup R_1 \quad \dim{R_1}=1 \; \textrm{(or $R_1$ is empty)},\\
W_0=Z\cup R_0 \quad \dim{R_0}=0 \; \textrm{(or $R_0$ is empty)}.\\
\end{align*}
The degrees of the Segre classes $s_0,s_1,s_2$ of $Z$ in $\PP^k$ are
computed as follows:
\begin{align*}
&\deg{s_0}=m^{k-2}-\deg{R_2},\\
&\deg{s_1}=m^{k-1}-\deg{R_1}-(k-1)m\deg{s_0},\\
&\deg{s_2}=m^k-\deg{R_0}-\binom{k}{2}m^2\deg{s_0}-km\deg{s_1}.
\end{align*}
\end{ex}

\section{Implementation and benchmarks} \label{sec:benchmarks}
There is an implementation of \algref{alg:segre} in the symbolic setting
using the software system \emph{Macaulay2} \cite{GS}. It uses the
improvement of \remref{rem:improvement}. The user may choose to be
given the degrees of the Chern-Fulton classes as output.
As an alternative to the Gr\"obner basis computations carried out in
\emph{Macaulay2}, one can use the regenerative cascade algorithm
\cite{HSW} implemented in the software package \emph{Bertini}
\cite{BHSW}. The regenerative cascade algorithm uses numerical
homotopy methods to collect data about solution sets of polynomial
equations and this data includes the degrees of the residuals that are
used in \algref{alg:segre} to compute the degrees of the Segre
classes. Both implementations are
available at http://www.math.su.se/$\sim$jost/segreimplementation.htm. 

\tabref{tab:benchmarks} shows run times on some examples, comparing
our implementation \mbox{``segreClass"} to two other algorithms. One is
the April 2009 version of ``CSM" which implements Aluffi's algorithm to
compute Segre classes, see \cite{A}. The other is the routine ``euler"
from \emph{Macaulay2} which computes the topological Euler
characteristic of a smooth projective variety. Observe that the input
to ``euler" is a projective variety, not an ideal. Following Table \ref{tab:benchmarks}, we provide a few details about each example. Additional details on how to
generate the equations for each example may be found at
http://www.math.su.se/$\sim$jost/segreimplementation.htm.

\begin{table}[h] 
\caption{Comparison of run times. On an AMD Athlon 64 Processor, 2.2
  GHz, and with 1 GB RAM. The computations marked with ``-" were
  terminated after 3 hours.}
\label{tab:benchmarks} 
\begin{tabular}{lllllllllll}
\hline 
Input & segreClass & CSM & euler \\
\hline
Rational normal curve in $\PP^6$ & 0.5s & 180s & 4s\\
Rational normal curve in $\PP^{10}$ & - & - & 512s\\
Grassmannian $\mathbb{G}(1,5) \subseteq \PP^{14}$ & - & 2s & -\\
Smooth surface in $\PP^8$ defined by minors & 89s & - & -\\
Abelian surface in $\PP^4$ & 175s & - & -\\
Segre embedding of $\PP^2 \times \PP^3$ in $\PP^{11}$ & - & 8s  & -\\
\hline
\end{tabular}
\end{table}

The defining equations of the rational normal curves in
\tabref{tab:benchmarks} are given as $(2\times2)$-minors of a matrix
with variables as entries. The Grassmann manifold is embedded with the
Pl\"ucker embedding. The surface in $\PP^8$ is defined by the
$(2\times2)$-minors of $(4\times3)$-matrix of random linear forms. The
ideal of the Abelian surface is generated in degrees 5 and 6. The
Grassmannian and the Segre product were run over $\mathbb{Q}$ and the
other examples were run over the finite field with $32749$ elements.

\begin{rem}
In connection with the comparison made in \tabref{tab:benchmarks} it
should be noted that the routine ``euler" computes the topological
Euler characteristic by first computing the Hodge numbers of the
variety and then taking an alternating sum of them. Thus, ``euler"
computes interesting information that is not attainable from the Segre
classes in any obvious way.
\end{rem}

\section{Conclusions}
This paper presents an elementary algorithm, based on residual intersection, to compute the degrees of Segre classes of a subscheme of projective space. 
The symbolic version of the algorithm has been implemented in \emph{Macaulay2} \cite{GS}. The numerical version, using numerical homotopy methods and the regenerative cascade algorithm \cite{HSW}, has been implemented in the software package \emph{Bertini} \cite{BHSW}.  The table of example run times illustrate the complementary nature
of the symbolic implementation of the algorithm to previous symbolic algorithms, in particular to the algorithm of Aluffi \cite{A} in the general case and to the algorithm ``euler" found in \emph{Macaulay2} when run on smooth projective varieties. The numeric implementation shows promise for extending the range of problems to which the algorithm can be applied.

\end{document}